\documentclass[a4paper,11pt]{amsart}
\usepackage{amsmath,amssymb}
\usepackage{microtype}
\usepackage{needspace}
\usepackage[textwidth=14.5cm,textheight=21.5cm,centering]{geometry}
\usepackage{cite}
\usepackage{xcolor}
\usepackage[colorlinks=true,linkcolor=blue,citecolor=blue,urlcolor=blue]{hyperref}

\allowdisplaybreaks[4]

\newtheorem{theorem}{Theorem}[section]
\newtheorem{lemma}[theorem]{Lemma}
\newtheorem{remark}[theorem]{Remark}
\newtheorem{proposition}[theorem]{Proposition}

\numberwithin{equation}{section}
\DeclareMathOperator{\tr}{tr}
\DeclareMathOperator{\Vol}{Vol}
\newcommand{\sphere}{\mathbb S}

\title[The Second Gap in Chern's Conjecture with Constant Cubic Trace]{The Second Gap and Rigidity in Chern's Conjecture with Constant Cubic Trace}

\author[J. Q. Ge]{Jianquan Ge}
\address{School of Mathematical Sciences, Laboratory of Mathematics and Complex Systems, Beijing Normal University, Beijing 100875, P. R.
CHINA.}
\email{jqge@bnu.edu.cn}

\author[H. X. Tan]{Huixin Tan\(^{*}\)}
\address{\(^{*}\)School of Mathematical Sciences, Laboratory of Mathematics and Complex Systems, Beijing Normal University, Beijing 100875, P. R.
CHINA.}
\email{hxtan@mail.bnu.edu.cn}

\author[W. J. Yan]{Wenjiao Yan}
\address{School of Mathematical Sciences, Laboratory of Mathematics and Complex Systems, Beijing Normal University, Beijing 100875, P. R.
CHINA.}
\email{wjyan@bnu.edu.cn}

\author[Y. H. Zhang]{Yunheng Zhang}
\address{School of Mathematical Sciences, Laboratory of Mathematics and Complex Systems, Beijing Normal University, Beijing 100875, P. R.
CHINA.}
\email{yunheng@mail.bnu.edu.cn}

\thanks{\(^{*}\) Corresponding author.}
\thanks{The project is partially supported by NSFC (No. 12271038, 12571049, 12526205) and the Fundamental Research Funds for the Central Universities.}

\begin{document}
\begin{abstract}
	Let \(M^n\subset\sphere^{n+1}(1)\) be a closed minimal hypersurface with constant \(S=|h|^2\) and constant \(f_3=\tr h^3\), where \(h\) is the shape operator.
	We prove that \(S>n\) implies \(S\geq2n\), and that the equality images are precisely the minimal Cartan isoparametric hypersurfaces with three principal curvatures, which occur only in dimensions \(n=3, 6, 12, 24\).
\end{abstract}

\maketitle

\section{Introduction}

Let \(M^n\subset\sphere^{n+1}(1)\) be a closed minimal hypersurface, and let \(h\) denote its shape operator.
Write \(S=|h|^2=\tr h^2\).
Since the Gauss equation gives \(R=n(n-1)-S\), constancy of the scalar curvature \(R\) is equivalent to constancy of \(S\).
Chern's conjecture asks whether, for each fixed dimension \(n\), the possible constant values of \(S\) form a discrete set \cite{chern_minimal_1968,yau_problem_1982}.
Its stronger form asserts that every closed minimal hypersurface with constant scalar curvature in a sphere is isoparametric, that is, all its principal curvatures are constant.
For a minimal isoparametric hypersurface with \(g\) distinct principal curvatures, one has \(S=(g-1)n\), where \(g\in\{1,2,3,4,6\}\) \cite{cartan_isoparametric_1939,munzner_isoparametric_1980,munzner_isoparametric_1981}.
These examples suggest the possible values \(0,n,2n,3n,5n\).
For background and related developments, we refer to \cite{ge_tang_2012,scherfner_weiss_yau_2012,ge_qian_tang_yan_2025}.

The first gap is given by Simons' theorem \cite{simons_minimal_1968}: if \(0\leq S\leq n\), then \(S\equiv0\) or \(S\equiv n\).
The first case is the totally geodesic sphere, and the second was classified independently by Chern--do Carmo--Kobayashi \cite{chern_minimal_1970} and Lawson \cite{lawson_local_1969} as the Clifford hypersurfaces
\[
	\sphere^k\left(\sqrt{\frac{k}{n}}\right)
	\times\sphere^{n-k}\left(\sqrt{\frac{n-k}{n}}\right),
	\qquad 1\leq k\leq n-1.
\]
The next question is whether a constant value \(S>n\) must satisfy \(S\geq2n\), and whether equality characterizes the minimal Cartan hypersurfaces with three principal curvatures.
The second gap therefore concerns both the possible values of \(S\) and the geometry of the hypersurfaces realizing the equality.
In this paper, we answer both questions under the additional assumption that \(f_3=\tr h^3\) is constant.

Peng--Terng \cite{peng_minimal_1983} proved that \(S>n\) implies \(S>n+\frac{1}{12n}\), and obtained the sharp estimate \(S\geq6\) in dimension \(3\).
Chang \cite{chang_minimal_1993} subsequently completed the characterization in that dimension.
In higher dimensions, Yang--Cheng \cite{yang_cheng_1998} enlarged the additive gap to \(\frac{n}{3}\), and Suh--Yang \cite{suh_yang_2007} further enlarged it to \(\frac{3n}{7}\).
A related problem replaces constancy of \(S\) by a pointwise pinching condition; see \cite{peng_scalar_1983,cheng_ishikawa_1999,wei_xu_2007,zhang_pinching_2010,ding_xin_2011,xu_xu_2017}.
In particular, Lei--Xu--Xu \cite{lei_xu_xu_2017} showed that \(0\leq S-n\leq\frac{n}{18}\) forces \(M\) to be a Clifford hypersurface.

Tang--Wei--Yan \cite{tang_wei_yan_2020} and Tang--Yan \cite{tang_yan_2023} established an isoparametric criterion for closed hypersurfaces in spheres.
For \(n>3\), constancy of \(f_k=\tr h^k\), \(1\leq k\leq n-1\), together with nonnegative scalar curvature, implies isoparametricity.
He--Xu--Zhao \cite{he_xu_zhao_2026} proved that a closed minimal hypersurface in \(\sphere^5(1)\) with constant scalar curvature and constant \(f_3\) is isoparametric.
Ge--Liu--Luo--Yan \cite{ge_liu_luo_yan_2026} obtained the same conclusion with constant Gauss--Kronecker curvature in place of constant \(f_3\).

Assuming that \(f_3\) is constant, Yang--Cheng \cite{yang_cheng_1998} proved that \(S>n\) implies \(S\geq\frac{5}{3}n\) for \(n\geq4\).
Cheng--Wei--Yamashiro \cite{cheng_second_2025} improved the estimate for \(n\geq5\) to
\[
	S>1.8252n-0.712898.
\]
Related estimates were obtained for minimal Willmore hypersurfaces with constant scalar curvature:
Deng--Gu--Wei \cite{deng_gu_wei_2017} proved the isoparametric classification in dimension four, while Ge--Tan--Yan--Zhang \cite{ge_tan_yan_zhang_2025} obtained, for \(n\geq5\),
\[
	S\geq n+\frac{4n+9-\sqrt{4n^2+60n+81}}{2}
	\qquad\text{if }S>n.
\]

For the original Chern's conjecture, Tan--Tang--Xie--Yan \cite{tan_tang_xie_yan_2026} recently proved that, in each fixed dimension, the values of \(S\) realized by closed embedded minimal hypersurfaces with constant \(S\) and constant \(f_3\) form a locally finite set.
Neither the topology of the hypersurface nor the constant value of \(f_3\) is fixed.
This establishes a stronger version of discreteness in that class, while the location of the curvature values and the classification of the corresponding hypersurfaces require further arguments.
The present paper determines the sharp second-gap bound under the same curvature assumptions and gives a complete classification at its endpoint.
Our results apply to closed immersions, without an embeddedness assumption.

We consider connected immersed hypersurfaces \(M\subset\mathbb S^{n+1}(1)\).
If \(M\) is one-sided, we work on its two-sided cover and assume that \(f_3\) is constant with respect to a fixed unit normal.
The conclusions about \(S\), the identity \(h^3=3h\), and the immersion are invariant under normal reversal.

\begin{theorem}\label{thm:main}
	Let \(M^n\), \(n\geq5\), be a closed connected minimal hypersurface immersed in \(\sphere^{n+1}(1)\).
	Suppose that \(S=|h|^2\) and \(f_3=\tr h^3\) are constant.
	If \(S>n\), then \(S\geq2n\).

	Equality holds if and only if the image is congruent to a minimal Cartan isoparametric hypersurface with three principal curvatures.
	In this case, the immersion is an embedding and \(n\in\{6,12,24\}\).
\end{theorem}

Thus, under the constant cubic trace assumption, both the second gap predicted by Chern's conjecture and the isoparametric classification at \(S=2n\) hold.
Together with the first gap theorem, this result classifies all immersed images with constant \(S\leq2n\) in this class: they are the totally geodesic sphere, the Clifford hypersurfaces, and the minimal Cartan hypersurfaces with \(g=3\).
Moreover, if \(n\geq5\) and \(n\notin\{6,12,24\}\), then \(S>n\) implies \(S>2n\).

To describe the equality examples, let \(\mathbb F\in\{\mathbb R,\mathbb C,\mathbb H,\mathbb O\}\) and \(m=\dim_{\mathbb R}\mathbb F\in\{1,2,4,8\}\).
The minimal member of Cartan's \(g=3\) family is the tube of radius \(\frac{\pi}{6}\) about the canonical minimal embedding \(\mathbb F P^2\subset\sphere^{3m+1}(1)\) \cite{cartan_isoparametric_1939,cecil_ryan_2015}.
It has dimension \(n=3m\) and principal curvatures \(\sqrt3,0,-\sqrt3\), each of multiplicity \(m\), so \(S=2n\), \(f_3=0\), and \(h^3=3h\).
For \(n\geq5\), the equality examples therefore correspond to \(\mathbb CP^2\), \(\mathbb HP^2\), and \(\mathbb OP^2\).

\begin{remark}\label{rem:low-dimensions}
	The restriction \(n\geq5\) in Theorem~\ref{thm:main} is made only to exclude the low-dimensional cases that are already understood.
	When \(n=3\), Chang \cite{chang_minimal_1993} proved that every closed minimal hypersurface in \(\sphere^4(1)\) with constant scalar curvature is isoparametric.
	In particular, its possible values of \(S\) are \(0\), \(3\), and \(6\), so \(S>3\) already implies the sharp second-gap estimate \(S\geq6=2n\).
	When \(n=4\), He--Xu--Zhao \cite{he_xu_zhao_2026} proved that a closed minimal hypersurface in \(\sphere^5(1)\) with constant scalar curvature and constant third mean curvature is isoparametric; under minimality, the latter condition is equivalent to the constancy of \(f_3\).
	The present argument also applies to \(n=4\) after adjusting the coarse numerical bounds in Lemma~\ref{lem:algebraic-positivity}.
	Use \(\frac{1}{n}\leq\frac{1}{4}\), replace \(\frac{28}{5}\) in \eqref{eq:A2-bound} by \(6\), and use \((n+1)(n+3)\leq\frac{35}{16}n^2\) in the endpoint estimate for \(P_c\).
	In the final estimate for \(\Theta_1\), use
	\[
		15n^3-148n>5n^3,\qquad
		102n^3+4n(93n+290)<268n^3,\qquad
		32n^2(n+1)\leq40n^3.
	\]
	These substitutions preserve strict positivity at both endpoints; the other estimates remain valid for \(n\geq4\).
	Thus the second-gap estimate in Theorem~\ref{thm:main} also holds in dimension four.
\end{remark}

By Yang–Cheng's estimate, it suffices to consider the range \(\frac{5}{3}n\leq S\leq2n\).
We first derive a lower bound for the squared norm of the symmetrized second covariant derivative of \(h\) by completing squares.
The remaining cubic expression is estimated in \(L^2\) using an integrated Weitzenböck identity for a cubic polynomial in \(h\) with constant coefficients.
Combining these estimates reduces the problem to a scalar inequality.
The resulting algebraic positivity and the \(L^2\)-estimate force \(S=2n\), \(f_3=0\), and \(h^3=3h\).
Together with minimality, these identities determine the three principal curvatures and their common multiplicity.
Cartan's classification identifies the image, and a covering argument shows that the immersion is an embedding.

Section~\ref{sec:curvature-identities} establishes the curvature identities and the \(L^2\)-estimate for the cubic expression.
Section~\ref{sec:fourth-order-comparison} combines these estimates to derive a global integral inequality. 
Section~\ref{sec:algebraic-estimates} proves the algebraic inequalities and Proposition~\ref{prop:closed-interval-rigidity}.
Section~\ref{sec:proof-main} completes the proof of Theorem~\ref{thm:main}, including the classification of the equality case and the embeddedness of the immersion.

\section{Curvature identities and preliminary lemmas}
\label{sec:curvature-identities}

Throughout this paper, all geometric quantities are taken with respect to the induced metric on \(M\).
The volume element is omitted from the notation in all integrals.
Let \(\{e_i\}_{i=1}^n\) be a local orthonormal frame, and write \(h_{ijk}=\nabla_kh_{ij}\) and \(h_{ijkl}=\nabla_l\nabla_kh_{ij}\).
The Gauss equation, the Codazzi equation, and the Ricci formula are
\begin{equation}\label{eq:gauss-ricci}
	\begin{aligned}
		R_{ijkl}&=\delta_{ik}\delta_{jl}-\delta_{il}\delta_{jk}
		+h_{ik}h_{jl}-h_{il}h_{jk},\\
		h_{ijk}&=h_{ikj},\\
		h_{ijkl}-h_{ijlk}&=\sum_m h_{im}R_{mjkl}+\sum_m h_{mj}R_{mikl}.
	\end{aligned}
\end{equation}
At each point, choose an orthonormal frame with \(h_{ij}=\lambda_i\delta_{ij}\), and set
\[
	f_k=\sum_i\lambda_i^k,\qquad
	A=\sum_{i,j,k}\lambda_i^2h_{ijk}^2,\qquad
	B=\sum_{i,j,k}\lambda_i\lambda_jh_{ijk}^2,\qquad
	C=\sum_{i,j,k}\lambda_i h_{ijk}^2.
\]
Throughout the proof, we assume that \(S>n\) and set
\begin{equation}\label{eq:gap-parameter}
	t=\frac{S-n}{S}.
\end{equation}
Reversing the unit normal if necessary, we may assume that \(f_3\geq0\).

The Simons identity \cite{simons_minimal_1968} and the constancy of \(S\) give
\begin{equation}\label{eq:nabla-h}
	\sum_{i,j,k}h_{ijk}^2=S(S-n)=tS^2.
\end{equation}
Differentiating the constant functions \(\tr h\), \(S\), and \(f_3\) yields
\begin{equation}\label{eq:derivative-traces}
	\sum_i h_{iik}=0,\qquad
	\sum_i\lambda_i h_{iik}=0,\qquad
	\sum_i\lambda_i^2h_{iik}=0.
\end{equation}
Since \(f_3\) is constant, the identity \(\Delta f_3=3(n-S)f_3+6C\) gives
\begin{equation}\label{eq:C}
	C=\frac{1}{2}(S-n)f_3=\frac{1}{2}tSf_3.
\end{equation}
Since \(h\) is trace-free, the metric and \(h\) are orthogonal.
Subtracting the corresponding components from \(h^2\) gives a symmetric two-tensor orthogonal to both.
Following the notation in \cite{yang_cheng_1998}, define its components \(f_{ij}\) and its squared norm \(f\) by
\begin{equation}\label{eq:yang-cheng-f}
	\begin{aligned}
		f_{ij}&=\sum_p h_{ip}h_{pj}-\frac{f_3}{S}h_{ij}-\frac{S}{n}\delta_{ij},\\
		f&=\sum_{i,j}f_{ij}^2=f_4-\frac{f_3^2}{S}-\frac{S^2}{n}.
	\end{aligned}
\end{equation}
Since \(\sum\limits_i\lambda_i=0\), \(\sum\limits_i\lambda_i^2=S\), and \(\sum\limits_i\lambda_i^3=f_3\), the definition in \eqref{eq:yang-cheng-f} gives
\begin{equation}\label{eq:f-orthogonality}
	\sum_i f_{ii}=0,\qquad
	\sum_{i,j}h_{ij}f_{ij}=0,\qquad
	\sum_i\lambda_i^2f_{ii}=f.
\end{equation}
The identity in \cite[(3.6)]{cheng_second_2025} gives
\begin{equation}\label{eq:Aminus2B}
	A-2B=Sf_4-f_3^2-S^2
	=Sf+\frac{tS^2}{1-t}.
\end{equation}
The constancy of \(S\) and \(f_3\), together with the definition of \(f\) in \eqref{eq:yang-cheng-f}, yields
\begin{equation}\label{eq:lapf}
	\frac{1}{4}\Delta f=\frac{1}{4}\Delta f_4=-tS\left(f+\frac{f_3^2}{S}+\frac{S^2}{n}\right)+2A+B.
\end{equation}
Solving \eqref{eq:Aminus2B} and \eqref{eq:lapf} for \(A\) and \(B\) gives
\begin{equation}\label{eq:AB}
	\begin{aligned}
		A&=\frac{(2t+1)Sf+2tf_3^2+\frac{3}{n}tS^3}{5}+\frac{1}{10}\Delta f,\\
		B&=\frac{(t-2)Sf+tf_3^2-\frac{1}{n}tS^3}{5}+\frac{1}{20}\Delta f.
	\end{aligned}
\end{equation}

By \eqref{eq:nabla-h}, \eqref{eq:C}, and the symmetry of \(h_{ijk}\), the weighted mean of \(\lambda_i+\lambda_j+\lambda_k\) with weights \(h_{ijk}^2\) is \(\frac{3C}{tS^2}=\frac{3f_3}{2S}\).
Expanding the corresponding centered square yields
\begin{equation}\label{eq:centered-square}
	A+2B-\frac{3}{4}tf_3^2
	=\frac{1}{3}\sum_{i,j,k}\left(\lambda_i+\lambda_j+\lambda_k-\frac{3f_3}{2S}\right)^2h_{ijk}^2\geq0.
\end{equation}
Substituting \eqref{eq:AB} into \eqref{eq:centered-square} yields
\begin{equation}\label{eq:centered-square-laplacian}
	\frac{5}{3}\sum_{i,j,k}\left(\lambda_i+\lambda_j+\lambda_k-\frac{3f_3}{2S}\right)^2h_{ijk}^2
	=(4t-3)Sf+\frac{1}{4}tf_3^2+\frac{tS^3}{n}+\Delta f.
\end{equation}
Let \(p\) be a maximum point of \(f\).
Then \(\Delta f(p)\leq0\), so for \(t<\frac{3}{4}\), \eqref{eq:centered-square-laplacian} implies
\begin{equation}\label{eq:fp-bound}
	0\leq f(p)\leq\frac{tS}{3-4t}\left(\frac{S}{n}+\frac{f_3^2}{4S^2}\right).
\end{equation}

By \cite[Theorem~2]{yang_cheng_1998}, the constancy of \(S\) and \(f_3\) and the assumption \(S>n\) imply \(S-n\geq\frac{2}{3}n\).
Since \(S-n=tS\), this reduces the interval \(n<S<2n\) to \(\frac{2}{5}\leq t<\frac{1}{2}\).
To include the equality case, we work from now on in the closed range \(\frac{5}{3}n\leq S\leq2n\), or equivalently \(\frac{2}{5}\leq t\leq\frac{1}{2}\).

By \eqref{eq:fp-bound}, the function
\begin{equation}\label{eq:psi-definition}
	\psi:=\frac{tS}{3-4t}\left(\frac{S}{n}+\frac{f_3^2}{4S^2}\right)-f
\end{equation}
is nonnegative on \(M\).
The next lemma expresses the integral of the centered square in \eqref{eq:centered-square} in terms of \(\int_M\psi\).
We will use this identity in Proposition~\ref{prop:hessian-cancellation}.
\begin{lemma}\label{lem:defect-evolution}
	Let \(M^n\subset\sphere^{n+1}(1)\) be a closed minimal hypersurface such that \(S\) and \(f_3\) are constant.
	Assume that \(\frac{5}{3}n\leq S\leq2n\), equivalently \(\frac{2}{5}\leq t\leq\frac{1}{2}\).
	Then
	\begin{equation}\label{eq:psi-range-integral}
		\begin{aligned}
			&0\leq\psi\leq\frac{tS}{3-4t}\left(\frac{S}{n}+\frac{f_3^2}{4S^2}\right),\\
			&\int_M\sum_{i,j,k}\left(\lambda_i+\lambda_j+\lambda_k-\frac{3f_3}{2S}\right)^2h_{ijk}^2
			=\frac{3S(3-4t)}{5}\int_M\psi.
		\end{aligned}
	\end{equation}
\end{lemma}

\begin{proof}
	Since \(\Delta\psi=-\Delta f\), substituting \eqref{eq:psi-definition} into \eqref{eq:centered-square-laplacian} yields
	\begin{equation}\label{eq:psi-equation}
		\Delta\psi=S(3-4t)\psi
		-\frac{5}{3}\sum_{i,j,k}\left(\lambda_i+\lambda_j+\lambda_k-\frac{3f_3}{2S}\right)^2h_{ijk}^2.
	\end{equation}
	The bounds for \(\psi\) follow from \eqref{eq:fp-bound} and \(f\geq0\).
	Integrating \eqref{eq:psi-equation} over the closed manifold \(M\) gives the integral identity in \eqref{eq:psi-range-integral}.
\end{proof}
Define
\begin{equation}\label{def:u}
	u_{ijkl}:=\frac{1}{4}(h_{ijkl}+h_{jkli}+h_{klij}+h_{lijk}).
\end{equation}
The Codazzi and Ricci identities \eqref{eq:gauss-ricci} imply that \(u\) is fully symmetric and that \(\sum\limits_k u_{ijkk}=-\frac{1}{2}(S-n)h_{ij}\).
The identity of Cheng, Wei, and Yamashiro \cite[Equation~(3.8)]{cheng_second_2025} is
\begin{equation}\label{C-W-Y}
	S(S-n)(S-2n)=\sum_{i,j,k,l}u_{ijkl}^2+\frac{3}{2}S(S-n)-\frac{3}{2}(A-2B).
\end{equation}
Substituting \eqref{eq:gap-parameter}, \eqref{eq:yang-cheng-f}, and \eqref{eq:Aminus2B} into \eqref{C-W-Y} yields
\begin{equation}\label{eq:u-norm}
	\sum_{i,j,k,l}u_{ijkl}^2
	=t(2t-1)S^3+\frac{3}{2}Sf+\frac{3t^2S^3}{2n}.
\end{equation}
Since \(S>n\), both \(nf_3^2+2S^3\) and \(nf_3^2+4S^3\) are positive.
To estimate \(\sum\limits_{i,j,k,l}u_{ijkl}^2\), define the symmetric two-tensor
\begin{equation}\label{eq:q-definition}
	\begin{aligned}
		q_{ij}
		&:=\sum_p f_{ip}h_{pj}
		-\frac{f_3}{4S}\left(\sum_p h_{ip}h_{pj}-\frac{S}{n}\delta_{ij}\right)+\frac{tS(nf_3^2+2S^3)}{2(nf_3^2+4S^3)}h_{ij}\\
		&=\sum_{p,m}h_{im}h_{mp}h_{pj}
		-\frac{5f_3}{4S}\sum_p h_{ip}h_{pj}+\left(\frac{tS(nf_3^2+2S^3)}{2(nf_3^2+4S^3)}-\frac{S}{n}\right)h_{ij}
		+\frac{f_3}{4n}\delta_{ij}.
	\end{aligned}
\end{equation}
Contracting \eqref{eq:q-definition} and using \eqref{eq:yang-cheng-f} gives
\begin{equation}\label{eq:q-contractions}
	\sum_i q_{ii}=0,\qquad
	\sum_{i,j}h_{ij}q_{ij}
	=f-\frac{f_3^2}{4S}
	+\frac{tS^2(nf_3^2+2S^3)}{2(nf_3^2+4S^3)}.
\end{equation}
Using \eqref{eq:C}, \eqref{eq:Aminus2B}, and \eqref{eq:AB}, together with the symmetry of \(u\), we obtain the contractions
\begin{align}
    \sum_{i,j,k,l}u_{ijkl}(f_{ij}h_{kl}+h_{ij}f_{kl})
    =&-2A+Sf+\frac{2tS^2}{1-t}+\frac{2f_3C}{S}\notag\\
    =&\frac{(3-4t)Sf+tf_3^2+\frac{4tS^3}{n}-\Delta f}{5},\notag\\
    \sum_{i,j,k,l}u_{ijkl}h_{ij}h_{kl}
    =&-C,\label{eq:projection-contractions}\\
    \sum_{i,j,k,l}u_{ijkl}(h_{ij}\delta_{kl}+h_{kl}\delta_{ij})
    =&-tS^2,\notag\\
    \sum_{i,j,k,l}u_{ijkl}\sum_{\mathrm{cyc}(j,k,l)}
    \bigl(\delta_{ij}q_{kl}+\delta_{kl}q_{ij}\bigr)
    =&-3tS\sum_i\lambda_iq_{ii},\notag
\end{align}
where \(\sum\limits_{\mathrm{cyc}(j,k,l)}\) denotes the sum over the ordered triples \((j,k,l)\), \((k,l,j)\), and \((l,j,k)\), with \(i\) fixed.
In the sequel, we continue to use this notation.

We estimate \(\sum\limits_{i,j,k,l}u_{ijkl}^2\) by completing a square.
The definition \eqref{eq:q-definition} and the identity \(\sum\limits_k u_{ijkk}=-\frac{1}{2}(S-n)h_{ij}\) imply that the symmetric four-tensor inside braces below is trace-free.
The nonnegativity of its squared norm gives
\begin{equation}\label{eq:component-square}
	\begin{aligned}
		0\leq\sum_{i,j,k,l}\Bigg\{&u_{ijkl}
		-\frac{nf_3^2+4S^3}{4(nf_3^2+2S^3)}
		\Bigg(\sum_{\mathrm{cyc}(j,k,l)}
		\bigl(f_{ij}h_{kl}+f_{kl}h_{ij}\bigr)\\
		&-\frac{4}{n+4}
		\sum_{\mathrm{cyc}(j,k,l)}
		\bigl(\delta_{ij}q_{kl}+\delta_{kl}q_{ij}\bigr)\Bigg)\\
		&+\frac{f_3(nf_3^2+4S^3)}{8S(nf_3^2+2S^3)}
		\sum_{\mathrm{cyc}(j,k,l)}
		\left(h_{ij}h_{kl}
		-\frac{2S}{n(n+2)}\delta_{ij}\delta_{kl}\right)
		\Bigg\}^2.
	\end{aligned}
\end{equation}
Expanding the squared norms and using the definitions of \(f_{ij}\) and \(q_{ij}\) in \eqref{eq:yang-cheng-f} and \eqref{eq:q-definition} yields
\begin{equation}\label{equations1}
	\begin{aligned}
		&\sum_{i,j,k,l}
		\Big(
		\sum_{\mathrm{cyc}(j,k,l)}
		\bigl(f_{ij}h_{kl}+f_{kl}h_{ij}\bigr)
		\Big)^2
		=6Sf+24\sum_i\lambda_i^2f_{ii}^2,\\
		&\sum_{i,j,k,l}
		\Big(
		\sum_{\mathrm{cyc}(j,k,l)}
		\big(
		h_{ij}h_{kl}
		-\frac{2S}{n(n+2)}\delta_{ij}\delta_{kl}
		\big)
		\Big)^2
		=3S^2+6f_4-\frac{12S^2}{n(n+2)},\\
		&\sum_{i,j,k,l}
		\Big(
		\sum_{\mathrm{cyc}(j,k,l)}
		\bigl(\delta_{ij}q_{kl}+\delta_{kl}q_{ij}\bigr)
		\Big)^2
		=6(n+4)\sum_iq_{ii}^2,\\
		&\sum_{i,j,k,l}\Big(
		\sum_{\mathrm{cyc}(j,k,l)}(\delta_{ij}h_{kl}+\delta_{kl}h_{ij})
		\Big)^2=6(n+4)S.
	\end{aligned}
\end{equation}
The mixed contractions are
\begin{equation}\label{equations2}
	\begin{aligned}
		&\sum_{i,j,k,l}
		\Big(
		\sum_{\mathrm{cyc}(j,k,l)}
		\bigl(f_{ij}h_{kl}+f_{kl}h_{ij}\bigr)
		\Big)
		\Big(
		\sum_{\mathrm{cyc}(j,k,l)}
		\big(
		h_{ij}h_{kl}
		-\frac{2S}{n(n+2)}\delta_{ij}\delta_{kl}
		\big)
		\Big)
		=12\sum_i f_{ii}\lambda_i^3,\\
		&\sum_{i,j,k,l}
		\Big(
		\sum_{\mathrm{cyc}(j,k,l)}
		\bigl(f_{ij}h_{kl}+f_{kl}h_{ij}\bigr)
		\Big)
		\Big(
		\sum_{\mathrm{cyc}(j,k,l)}
		\bigl(\delta_{ij}q_{kl}+\delta_{kl}q_{ij}\bigr)
		\Big)
		=24\sum_i\lambda_i f_{ii}q_{ii},\\
		&\sum_{i,j,k,l}
		\Big(
		\sum_{\mathrm{cyc}(j,k,l)}
		\big(
		h_{ij}h_{kl}
		-\frac{2S}{n(n+2)}\delta_{ij}\delta_{kl}
		\big)
		\Big)
		\Big(
		\sum_{\mathrm{cyc}(j,k,l)}
		\bigl(\delta_{ij}q_{kl}+\delta_{kl}q_{ij}\bigr)
		\Big)
		=12\sum_i\lambda_i^2q_{ii}.
	\end{aligned}
\end{equation}
The terms containing \(q_{ij}\) in the expansion of \eqref{eq:component-square} are
\[
	\begin{aligned}
		&\frac{6(nf_3^2+4S^3)^2}
		{(n+4)(nf_3^2+2S^3)^2}
		\left(
		\sum_iq_{ii}^2
		-2\sum_i\lambda_i f_{ii}q_{ii}
		+\frac{f_3}{2S}\sum_i\lambda_i^2q_{ii}
		\right)\\
		&-\frac{6tS(nf_3^2+4S^3)}
		{(n+4)(nf_3^2+2S^3)}
		\sum_i\lambda_iq_{ii}.
	\end{aligned}
\]
The definition \eqref{eq:q-definition} gives
\[
	q_{ii}-\lambda_i f_{ii}+\frac{f_3}{4S}\lambda_i^2
	=
	\frac{f_3}{4n}
	+\frac{tS(nf_3^2+2S^3)}
	{2(nf_3^2+4S^3)}\lambda_i .
\]
Completing the square and using this identity yields
\begin{equation}\label{eq:q-quadratic-contraction}
	\begin{aligned}
		&\sum_iq_{ii}^2-2\sum_i\lambda_i f_{ii}q_{ii}
		+\frac{f_3}{2S}\sum_i\lambda_i^2q_{ii}\\
		&=\sum_i\left(q_{ii}-\lambda_i f_{ii}
		+\frac{f_3}{4S}\lambda_i^2\right)^2
		-\sum_i\left(\lambda_i f_{ii}-\frac{f_3}{4S}\lambda_i^2\right)^2\\
		&=-\sum_i\lambda_i^2f_{ii}^2
		+\frac{f_3}{2S}\sum_i f_{ii}\lambda_i^3
		-\frac{f_3^2f_4}{16S^2}+\frac{f_3^2}{16n}
		+\frac{t^2S^3(nf_3^2+2S^3)^2}{4(nf_3^2+4S^3)^2}.
	\end{aligned}
\end{equation}
We substitute \eqref{eq:projection-contractions}, \eqref{equations1}, \eqref{equations2}, and \eqref{eq:q-quadratic-contraction} into \eqref{eq:component-square}.
Using \eqref{eq:q-contractions} and \(f_4=f+\frac{f_3^2}{S}+\frac{S^2}{n}\) then yields
\begin{equation}\label{eq:component-square-expansion}
	\begin{aligned}
		\sum_{i,j,k,l}u_{ijkl}^2\geq{}&
		\frac{3t^2S^3}{2(n+4)}+\frac{3(nf_3^2+4S^3)}
		{2(nf_3^2+2S^3)}
		\Bigg\{
		S\left(
		\frac{3-4t}{5}
		+\frac{4t}{n+4}
		\right)f
		+\frac{t(9n+16)}{20(n+4)}f_3^2\\
		&+\frac{4tS^3}{5n}
		-\frac{1}{5}\Delta f
		\Bigg\}-\frac{3n(nf_3^2+4S^3)^2}
		{2(n+4)(nf_3^2+2S^3)^2}
		\Bigg\{
		\sum_i\lambda_i^2f_{ii}^2
		-\frac{f_3}{2S}\sum_i f_{ii}\lambda_i^3\\
		&
		+\frac{(n+4)S}{4n}f
		+\frac{f_3^2}{32}
		\left(
		1+\frac{6}{n}
		+\frac{4}{n(n+2)}
		+\frac{2f}{S^2}
		+\frac{2f_3^2}{S^3}
		\right)
		\Bigg\}.
	\end{aligned}
\end{equation}
The expression \(\sum\limits_i\lambda_i^2f_{ii}^2-\frac{f_3}{2S}\sum\limits_i f_{ii}\lambda_i^3\) occurs with a negative coefficient in \eqref{eq:component-square-expansion}, so we need an upper bound for its integral over \(M\).
The following identities reduce this estimate to the \(L^2\)-bound in Proposition~\ref{prop:hessian-cancellation}, Young's inequality, and the upper bound for \(f\) in \eqref{eq:fp-bound}.

The definitions of \(f_{ij}\) and \(f\) in \eqref{eq:yang-cheng-f} give
\begin{equation}\label{eq:ob}
	\sum_i f_{ii}\lambda_i^3
	=\frac{f_3}{S}f+\sum_i\lambda_i f_{ii}^2.
\end{equation}
By \eqref{eq:f-orthogonality}, the symmetric two-tensor with components \(\sum\limits_p f_{ip}h_{pj}-\frac{f_3}{2S}f_{ij}\) has inner product \(f\) with \(h\).
Since \(|h|^2=S\), subtracting \(\frac{f}{S}h\) gives a tensor orthogonal to \(h\).
Completing the square and using \eqref{eq:ob} yields
\begin{equation}\label{eq:cubic-mixed-identity}
	\begin{aligned}
		\sum_i\lambda_i^2f_{ii}^2
		-\frac{f_3}{2S}\sum_i f_{ii}\lambda_i^3
		&=\sum_i\left[f_{ii}\left(\lambda_i-\frac{f_3}{2S}\right)
		-\frac{f}{S}\lambda_i\right]^2+\frac{f^2}{S}\\
		&\quad+\frac{f_3}{2S}\sum_i
		\left[f_{ii}\left(\lambda_i-\frac{f_3}{2S}\right)
		-\frac{f}{S}\lambda_i\right]f_{ii}
		-\frac{f_3^2}{2S^2}f.
	\end{aligned}
\end{equation}
Expanding the square and using \eqref{eq:f-orthogonality} gives
\begin{equation}\label{eq:cubic-square}
	\begin{aligned}
		\sum_i\left[f_{ii}\left(\lambda_i-\frac{f_3}{2S}\right)
		-\frac{f}{S}\lambda_i\right]^2
		&=\sum_{i,j}\left(
		\sum_p f_{ip}h_{pj}-\frac{f_3}{2S}f_{ij}-\frac{f}{S}h_{ij}
		\right)^2\\
		&=\sum_i\left(\lambda_i-\frac{f_3}{2S}\right)^2f_{ii}^2
		-\frac{f^2}{S},
	\end{aligned}
\end{equation}
where the first equality expresses the sum as the squared norm of a smooth tensor and hence shows that it is independent of the chosen principal frame.

To estimate \eqref{eq:cubic-square} globally, we shall apply a Weitzenb\"ock formula to a suitable symmetric two-tensor.
The tensor used below will be divergence-free but need not satisfy the Codazzi equation.
For a symmetric two-tensor \((b_{ij})\), write \(b_{ijk}=\nabla_kb_{ij}\), and its divergence has components \(\sum\limits_jb_{ijj}\).
We shall then derive the following integrated Weitzenb\"ock identity.

\begin{lemma}\label{lem:codazzi-hodge}
	Let \((b_{ij})\) be a smooth symmetric two-tensor on a closed Riemannian manifold \(M\).
	With the Ricci convention in \eqref{eq:gauss-ricci}, we have
	\begin{equation}\label{eq:integrated-weitzenbock}
		\begin{aligned}
			\int_M\Bigg\{\sum_{i,j,k}b_{ijk}^2
			-\sum_i\Big(\sum_j b_{ijj}\Big)^2
			-\frac{1}{2}\sum_{i,j,k}(b_{ikj}-b_{ijk})^2
			+\sum_{i<j}R_{ijij}(b_{ii}-b_{jj})^2\Bigg\}=0,
		\end{aligned}
	\end{equation}
	where the orthonormal frame diagonalizes \((b_{ij})\) at the point where the curvature term is evaluated.
	Equivalently,
	\begin{equation}\label{eq:component-weitzenbock}
		\begin{aligned}
			\int_M\sum_{i,j,k}b_{ijk}b_{ikj}
			=\int_M\sum_i\Big(\sum_j b_{ijj}\Big)^2
			-\int_M\sum_{i<j}R_{ijij}(b_{ii}-b_{jj})^2.
		\end{aligned}
	\end{equation}
\end{lemma}

\begin{proof}
	Expanding the difference of derivatives yields
	\begin{equation}\label{initial expansion}
		\frac{1}{2}\sum_{i,j,k}(b_{ikj}-b_{ijk})^2
		=\sum_{i,j,k}b_{ijk}^2-\sum_{i,j,k}b_{ijk}b_{ikj}.
	\end{equation}
	It therefore suffices to prove \eqref{eq:component-weitzenbock}.
	Since \(M\) is closed, integration by parts followed by the Ricci identity gives
	\[
		\begin{aligned}
			\int_M\sum_{i,j,k}b_{ijk}b_{ikj}
			={}&-\int_M\sum_{i,j,k}b_{ij}b_{ikjk}\\
			={}&-\int_M\sum_{i,j,k}b_{ij}b_{ikkj}
			-\int_M\sum_{i,j,k,m}b_{ij}
			\bigl(b_{im}R_{mkjk}+b_{mk}R_{mijk}\bigr),
		\end{aligned}
	\]
	where \(b_{ikjk}=\nabla_k\nabla_jb_{ik}\).
	A second integration by parts yields
	\[
		-\int_M\sum_{i,j,k}b_{ij}b_{ikkj}
		=\int_M\sum_i\Big(\sum_j b_{ijj}\Big)
		\Big(\sum_k b_{ikk}\Big)
		=\int_M\sum_i\Big(\sum_j b_{ijj}\Big)^2.
	\]
	For the curvature term, choose an orthonormal frame with \(b_{ij}=b_{ii}\delta_{ij}\) at the point under consideration.
	In this frame,
	\[
		\begin{aligned}
			\sum_{i,j,k,m}b_{ij}
			\bigl(b_{im}R_{mkjk}+b_{mk}R_{mijk}\bigr)
			=\sum_{i,k}b_{ii}(b_{ii}-b_{kk})R_{ikik}
			=\sum_{i<k}R_{ikik}(b_{ii}-b_{kk})^2.
		\end{aligned}
	\]
	The second equality pairs the summands indexed by \((i,k)\) and \((k,i)\).
	This proves \eqref{eq:component-weitzenbock}, and \eqref{initial expansion} gives \eqref{eq:integrated-weitzenbock}.
\end{proof}

We now apply Lemma~\ref{lem:codazzi-hodge} to a divergence-free cubic polynomial in \(h\) with constant coefficients.
The Gauss equation expresses the curvature term in terms of the squared norm in \eqref{eq:cubic-square} and \(\psi^2\).
Expanding the covariant derivative terms and using Lemma~\ref{lem:defect-evolution} yields the following \(L^2\)-estimate.
\begin{proposition}\label{prop:hessian-cancellation}
	Let \(M^n\subset\sphere^{n+1}(1)\) be a closed minimal hypersurface.
	Suppose that \(S\) and \(f_3\) are constant and \(\frac{5}{3}n\leq S\leq2n\).
	Then the expression in \eqref{eq:cubic-square} satisfies
	\begin{equation}\label{eq:sharp-residual-bound}
		\begin{aligned}
			&\int_M\sum_i\left[f_{ii}\left(\lambda_i-\frac{f_3}{2S}\right)-\frac{f}{S}\lambda_i\right]^2\leq\left(\frac{S}{n}+\frac{f_3^2}{4S^2}\right)
			\left(\frac{3}{5}+\frac{t^2}{(3-4t)(1-t)}\right)\int_M\psi,
		\end{aligned}
	\end{equation}
	where \(f_{ij}\) and \(f\) are defined in \eqref{eq:yang-cheng-f}, and \(\psi\) is defined in \eqref{eq:psi-definition}.
\end{proposition}

\begin{proof}
	Define the symmetric two-tensor \((r_{ij})\) by subtracting \(\frac{\psi}{S}h_{ij}\) from the tensor in \eqref{eq:cubic-square}:
	\begin{equation}\label{eq:k-polynomial}
		\begin{aligned}
			r_{ij}
			&:=\left(\sum_p f_{ip}h_{pj}-\frac{f_3}{2S}f_{ij}-\frac{f}{S}h_{ij}\right)-\frac{\psi}{S}h_{ij}\\
			&=\sum_p f_{ip}h_{pj}-\frac{f_3}{2S}f_{ij}-\frac{f+\psi}{S}h_{ij}\\
			&=\sum_{p,q}h_{ip}h_{pq}h_{qj}-\frac{3f_3}{2S}\sum_p h_{ip}h_{pj}+\frac{3\bigl((2-3t)f_3^2-4S^2\bigr)}{4S^2(3-4t)}h_{ij}+\frac{f_3}{2n}\delta_{ij}.
		\end{aligned}
	\end{equation}
	By \eqref{eq:psi-definition}, \(f+\psi=\frac{tS}{3-4t}\left(\frac{S}{n}+\frac{f_3^2}{4S^2}\right)\) is constant on \(M\).
	Thus \((r_{ij})\) is a cubic polynomial in \(h\) with constant coefficients.
	Equations~\eqref{eq:f-orthogonality} and~\eqref{eq:k-polynomial} yield
	\begin{equation}\label{eq:r-orthogonality}
		\begin{aligned}
			\sum_i r_{ii}&=0,\qquad
			\sum_{i,j}r_{ij}h_{ij}=-\psi,\\
			\sum_{i,j}r_{ij}^2
			&=\sum_i\left[f_{ii}\left(\lambda_i-\frac{f_3}{2S}\right)
			-\frac{f}{S}\lambda_i\right]^2+\frac{\psi^2}{S}.
		\end{aligned}
	\end{equation}

	At a fixed point, choose a local orthonormal frame that is geodesic at that point and diagonalizes \(h\) there.
	The product rule gives
	\[
		\begin{aligned}
			\nabla_l\left(\sum_{p,q}h_{ip}h_{pq}h_{qj}\right)
			={}&\sum_{p,q}h_{ipl}h_{pq}h_{qj}
			+\sum_{p,q}h_{ip}h_{pql}h_{qj}
			+\sum_{p,q}h_{ip}h_{pq}h_{qjl}\\
			={}&(\lambda_i^2+\lambda_i\lambda_j+\lambda_j^2)h_{ijl}.
		\end{aligned}
	\]
	Differentiating \eqref{eq:k-polynomial} therefore yields
	\begin{equation}\label{eq:k-derivative}
		r_{ijl}
		=\left(\lambda_i^2+\lambda_i\lambda_j+\lambda_j^2
		-\frac{3f_3}{2S}(\lambda_i+\lambda_j)
		+\frac{3\bigl((2-3t)f_3^2-4S^2\bigr)}{4S^2(3-4t)}\right)h_{ijl}.
	\end{equation}
	By \eqref{eq:derivative-traces} and the symmetry of \(h_{ijk}\), the divergence is
	\begin{equation}\label{eq:r-divergence}
		\begin{aligned}
			\sum_j r_{ijj}
			={}&\left(\lambda_i^2-\frac{3f_3}{2S}\lambda_i
			+\frac{3\bigl((2-3t)f_3^2-4S^2\bigr)}{4S^2(3-4t)}\right)\sum_j h_{ijj}\\
			&+\left(\lambda_i-\frac{3f_3}{2S}\right)\sum_j\lambda_jh_{ijj}
			+\sum_j\lambda_j^2h_{ijj}=0.
		\end{aligned}
	\end{equation}
	The symmetry of \(h_{ijk}\) and \eqref{eq:k-derivative} yield
	\begin{equation}\label{eq:r-codazzi-difference}
		r_{ikj}-r_{ijk}
		=(\lambda_k-\lambda_j)
		\left(\lambda_i+\lambda_j+\lambda_k-\frac{3f_3}{2S}\right)h_{ijk}.
	\end{equation}
	Thus \((r_{ij})\) is divergence-free, although it need not satisfy the Codazzi equation.

	We first evaluate the curvature term in Lemma~\ref{lem:codazzi-hodge}.
	The chosen frame also diagonalizes \((r_{ij})\), since \((r_{ij})\) is a polynomial in \(h\).
	The Gauss equation and \eqref{eq:r-orthogonality} now give
	\begin{equation}\label{eq:r-curvature}
		\begin{aligned}
			\sum_{i<j}R_{ijij}(r_{ii}-r_{jj})^2
			&=n\sum_{i,j}r_{ij}^2-\psi^2\\
			&=n\sum_i\left[f_{ii}\left(\lambda_i-\frac{f_3}{2S}\right)
			-\frac{f}{S}\lambda_i\right]^2-t\psi^2.
		\end{aligned}
	\end{equation}
	By \eqref{initial expansion}, \eqref{eq:r-divergence}, and \eqref{eq:r-curvature}, the identity \eqref{eq:component-weitzenbock} becomes
	\begin{equation}\label{eq:r-weitzenbock}
		\begin{aligned}
			&\int_M\left\{\sum\limits_{i,j,k}r_{ijk}^2-\frac{1}{2}\sum\limits_{i,j,k}(r_{ikj}-r_{ijk})^2\right\}\\
			&\quad+n\int_M\sum_i\left[f_{ii}\left(\lambda_i-\frac{f_3}{2S}\right)
			-\frac{f}{S}\lambda_i\right]^2
			=t\int_M\psi^2.
		\end{aligned}
	\end{equation}
	For the derivative terms, set \(x_i=\lambda_i-\frac{f_3}{2S}\) and \(\beta=\frac{1-t}{3-4t}\left(\frac{S}{n}+\frac{f_3^2}{4S^2}\right)\).
	Writing \(P(x,y)=x^2+xy+y^2-3\beta\), we can express \eqref{eq:k-derivative} as \(r_{ijk}=P(x_i,x_j)h_{ijk}\).
	Likewise, \eqref{eq:r-codazzi-difference} reads \(r_{ikj}-r_{ijk}=(x_k-x_j)(x_i+x_j+x_k)h_{ijk}\).
	By \eqref{initial expansion}, the derivative expression in \eqref{eq:r-weitzenbock} is therefore \(\sum\limits_{i,j,k}P(x_i,x_j)P(x_i,x_k)h_{ijk}^2\).
	For real \(x,y,z\), expanding the products gives
	\[
		\begin{aligned}
			&\frac{1}{3}\bigl(P(x,y)P(x,z)+P(y,z)P(y,x)+P(z,x)P(z,y)\bigr)\\
			=&\frac{1}{4}(x^2+y^2+z^2-6\beta)^2
			+\frac{1}{12}(x+y+z)^4-\beta(x+y+z)^2.
		\end{aligned}
	\]
	Since \(h_{ijk}^2\) is symmetric in \(i,j,k\), we may replace \(P(x_i,x_j)P(x_i,x_k)\) in the sum by its average over the three cyclic permutations.
	Moreover, \eqref{eq:yang-cheng-f} implies
	\[
		\begin{aligned}
			x_i^2+x_j^2+x_k^2-6\beta
			=f_{ii}+f_{jj}+f_{kk}
			+\frac{3(1-2t)}{3-4t}\left(\frac{S}{n}+\frac{f_3^2}{4S^2}\right).
		\end{aligned}
	\]
	Substitution in the cyclic average yields
	\begin{equation}\label{eq:k-derivative-product}
		\begin{aligned}
			\sum\limits_{i,j,k}r_{ijk}^2-\frac{1}{2}\sum\limits_{i,j,k}(r_{ikj}-r_{ijk})^2
			={}&\frac{1}{4}\sum_{i,j,k}
			\left(f_{ii}+f_{jj}+f_{kk}
			+\frac{3(1-2t)}{3-4t}\left(\frac{S}{n}+\frac{f_3^2}{4S^2}\right)\right)^2h_{ijk}^2\\
			&+\frac{1}{12}\sum_{i,j,k}
			\left(\lambda_i+\lambda_j+\lambda_k-\frac{3f_3}{2S}\right)^4h_{ijk}^2\\
			&-\frac{1-t}{3-4t}\left(\frac{S}{n}+\frac{f_3^2}{4S^2}\right)\sum_{i,j,k}
			\left(\lambda_i+\lambda_j+\lambda_k-\frac{3f_3}{2S}\right)^2h_{ijk}^2.
		\end{aligned}
	\end{equation}
	Substituting \eqref{eq:k-derivative-product} into \eqref{eq:r-weitzenbock} and using \eqref{eq:psi-range-integral}, we obtain the identity
	\begin{equation}\label{eq:exact-hessian-cancellation}
		\begin{aligned}
			&n\int_M\sum_i\left[f_{ii}\left(\lambda_i-\frac{f_3}{2S}\right)-\frac{f}{S}\lambda_i\right]^2\\
			+&\frac{1}{4}\int_M\sum_{i,j,k}
			\left(f_{ii}+f_{jj}+f_{kk}
			+\frac{3(1-2t)}{3-4t}\left(\frac{S}{n}+\frac{f_3^2}{4S^2}\right)\right)^2h_{ijk}^2\\
			+&\frac{1}{12}\int_M\sum_{i,j,k}
			\left(\lambda_i+\lambda_j+\lambda_k-\frac{3f_3}{2S}\right)^4h_{ijk}^2\\
			=&t\int_M\psi^2+\frac{3n}{5}\left(\frac{S}{n}+\frac{f_3^2}{4S^2}\right)\int_M\psi.
		\end{aligned}
	\end{equation}
	The second and third terms on the left-hand side of \eqref{eq:exact-hessian-cancellation} are nonnegative.
	Discarding them and applying the upper bound for \(\psi\) in \eqref{eq:psi-range-integral} gives
	\[
		\begin{aligned}
			n\int_M\sum_i\left[f_{ii}\left(\lambda_i-\frac{f_3}{2S}\right)-\frac{f}{S}\lambda_i\right]^2
			&\leq t\int_M\psi^2+\frac{3n}{5}\left(\frac{S}{n}+\frac{f_3^2}{4S^2}\right)\int_M\psi\\
			&\leq\left(\frac{t^2S}{3-4t}+\frac{3n}{5}\right)
			\left(\frac{S}{n}+\frac{f_3^2}{4S^2}\right)\int_M\psi.
		\end{aligned}
	\]
	Division by \(n=S(1-t)>0\) proves \eqref{eq:sharp-residual-bound}.
\end{proof}

\section{Global fourth-order estimate}
\label{sec:fourth-order-comparison}

We now integrate \eqref{eq:u-norm} and \eqref{eq:component-square-expansion} over \(M\) and compare the resulting expressions.
The Laplacian term has zero integral because \(M\) is closed.
Lemma~\ref{lem:defect-evolution}, Proposition~\ref{prop:hessian-cancellation}, and Young's inequality bound the integral of the left-hand side of \eqref{eq:cubic-mixed-identity}.
This comparison yields an inequality involving only \(\int_M f\), \(\int_M\psi\), and \(\Vol(M)\), with coefficients determined by the constants \(n\), \(S\), and \(f_3\).

\begin{lemma}\label{lem:integral-comparison}
	Let \(M^n\subset\sphere^{n+1}(1)\) be a closed minimal hypersurface such that \(S\) and \(f_3\) are constant.
	Suppose that \(\frac{5}{3}n\leq S\leq2n\), and define
	\begin{equation}\label{eq:Theta-definition}
		\begin{aligned}
			\Theta:=\int_M\Bigg\{&\Bigg[
			\frac{S(nf_3^2+2S^3)\bigl((n+4)(3-4t)+20t\bigr)}{5n(nf_3^2+4S^3)}
			-\frac{(n+4)S(nf_3^2+2S^3)^2}{n(nf_3^2+4S^3)^2} \\
			&-\frac{t}{3-4t}\left(\frac{S}{n}+\frac{f_3^2}{4S^2}\right)
			-\frac{(n+4)S}{4n}+\frac{5f_3^2}{16S^2}\Bigg]f\\
			&-\frac{3}{2}\left(\frac{3}{5}+\frac{t^2}{(3-4t)(1-t)}\right)
			\left(\frac{S}{n}+\frac{f_3^2}{4S^2}\right)\psi\\
			&+\frac{t(nf_3^2+2S^3)\bigl(n(9n+16)f_3^2+16(n+4)S^3\bigr)}{20n^2(nf_3^2+4S^3)}\\
			&-\frac{2(n+4)S^3(nf_3^2+2S^3)^2}{3n(nf_3^2+4S^3)^2}
			\left(t(2t-1)+\frac{6t^2}{n(n+4)}\right)\\
			&-\frac{f_3^2}{32}\left(1+\frac{6}{n}+\frac{4}{n(n+2)}
			+\frac{2f_3^2}{S^3}\right)\Bigg\},
		\end{aligned}
	\end{equation}
	where \(f\) and \(\psi\) are defined in \eqref{eq:yang-cheng-f} and \eqref{eq:psi-definition}, respectively.
	Then
	\begin{equation}\label{eq:Theta-sign}
		\Theta\leq0.
	\end{equation}
\end{lemma}

\begin{proof}
	We first estimate the mixed term on the right-hand side of \eqref{eq:cubic-mixed-identity}.
	Young's inequality and the identity \(\sum\limits_{i,j}f_{ij}^2=f\) yield
	\begin{equation}\label{eq:Young}
		\begin{aligned}
			\frac{f_3}{2S}\sum_i
			\left[f_{ii}\left(\lambda_i-\frac{f_3}{2S}\right)
			-\frac{f}{S}\lambda_i\right]f_{ii}
			\leq\frac{1}{2}\sum_i
			\left[f_{ii}\left(\lambda_i-\frac{f_3}{2S}\right)
			-\frac{f}{S}\lambda_i\right]^2
			+\frac{f_3^2}{8S^2}f.
		\end{aligned}
	\end{equation}
	Since \(f\geq0\) and \(\psi\geq0\), the definition \eqref{eq:psi-definition} also gives
	\begin{equation}\label{psif}
		\int_M f^2\leq\frac{tS}{3-4t}
		\left(\frac{S}{n}+\frac{f_3^2}{4S^2}\right)\int_M f.
	\end{equation}
	Integrating \eqref{eq:cubic-mixed-identity} and applying \eqref{eq:Young}, \eqref{psif}, and Proposition~\ref{prop:hessian-cancellation}, we obtain
	\begin{equation}\label{eq:scalar-integral-upper}
		\begin{aligned}
			\int_M\Big(\sum_i\lambda_i^2f_{ii}^2
			-\frac{f_3}{2S}\sum_i f_{ii}\lambda_i^3\Big)
			\leq&\frac{3}{2}\left(\frac{3}{5}+\frac{t^2}{(3-4t)(1-t)}\right)
			\left(\frac{S}{n}+\frac{f_3^2}{4S^2}\right)\int_M\psi\\
			&+\left[\frac{t}{3-4t}\left(\frac{S}{n}+\frac{f_3^2}{4S^2}\right)
			-\frac{3f_3^2}{8S^2}\right]\int_M f.
		\end{aligned}
	\end{equation}
	Integrating \eqref{eq:u-norm} over \(M\) yields
	\begin{equation}\label{eq:integral-u-F}
		\begin{aligned}
			\int_M\sum_{i,j,k,l}u_{ijkl}^2
			&=\frac{3}{2}S\int_M f
			+\int_M\left(t(2t-1)S^3+\frac{3t^2S^3}{2n}\right).
		\end{aligned}
	\end{equation}
	Finally, integrating \eqref{eq:component-square-expansion} over $M$ and using \eqref{eq:scalar-integral-upper} and \eqref{eq:integral-u-F}, we obtain \eqref{eq:Theta-sign}.
\end{proof}

\section{Algebraic estimates}
\label{sec:algebraic-estimates}
For \(\Theta\) defined in \eqref{eq:Theta-definition}, Lemma~\ref{lem:integral-comparison} gives
\[
	\Theta\leq0.
\]
We first record the parameter bounds needed to estimate \(\Theta\).

Since the principal curvatures satisfy \(\sum\limits_i\lambda_i=0\) and \(\sum\limits_i\lambda_i^2=S\), Okumura's inequality \cite{okumura_hypersurfaces_1974} yields
\[
	|f_3|=\Big|\sum_i\lambda_i^3\Big|
	\leq\frac{n-2}{\sqrt{n(n-1)}}S^{\frac{3}{2}}.
\]
Consequently,
\[
	\frac{f_3^2}{S^3}\leq\frac{(n-2)^2}{n(n-1)}\leq\frac{n-2}{n}.
\]
Throughout this section, the parameters satisfy
\begin{equation}\label{eq:certificate-domain}
	n\geq5,\qquad \frac{2}{5}\leq t\leq\frac{1}{2},\qquad
	0\leq\frac{f_3^2}{S^3}\leq\frac{n-2}{n}.
\end{equation}
The definition of \(\psi\) in \eqref{eq:psi-definition} and the bounds in \eqref{eq:psi-range-integral} give
\begin{equation}\label{eq:integral-f-psi}
	\begin{aligned}
		\int_M f&=\frac{tS}{3-4t}\left(\frac{S}{n}+\frac{f_3^2}{4S^2}\right)\Vol(M)-\int_M\psi,\\
		0&\leq\int_M\psi\leq\frac{tS}{3-4t}\left(\frac{S}{n}+\frac{f_3^2}{4S^2}\right)\Vol(M).
	\end{aligned}
\end{equation}
Substituting \eqref{eq:integral-f-psi} into \eqref{eq:Theta-definition} yields
\begin{equation}\label{eq:Theta-affine}
	\Theta=\Theta_0\Vol(M)+\Theta_1\int_M\psi,
\end{equation}
where \(\Theta_0\) and \(\Theta_1\) are constant on \(M\).
Their explicit expressions will be given in the proof of Lemma~\ref{lem:algebraic-positivity}. We shall prove that \(\Theta_0\geq0\), with equality exactly when \(t=\frac{1}{2}\) and \(f_3=0\), and that \(\Theta_1>0\).
These inequalities, together with \eqref{eq:Theta-sign} and Proposition~\ref{prop:hessian-cancellation}, yield the rigidity statement in Proposition~\ref{prop:closed-interval-rigidity}.

\begin{lemma}\label{lem:cubic-endpoint}
	Let \(p(s)=a_0+a_1s+a_2s^2-a_3s^3\), where \(a_0,a_1,a_3>0\) and \(a_2\in\mathbb R\).
	For every \(L>0\), the minimum of \(p\) on \([0,L]\) is attained at an endpoint.
\end{lemma}

\begin{proof}
	The discriminant of \(p'\) is \(4a_2^2+12a_1a_3>0\), and the product of its two zeros is
	\[
		-\frac{a_1}{3a_3}<0.
	\]
 Thus \(p'\) has a unique positive zero. Since \(p'(0)>0\) and the leading coefficient of \(p'\) is negative, the corresponding critical point of \(p\) is a local maximum.
Hence \(p\) has no interior minimum on \((0,L)\), and
\[
\min_{0\leq s\leq L}p(s)=\min\{p(0),p(L)\}.
\]
\end{proof}

\begin{lemma}\label{lem:algebraic-positivity}
	For parameters satisfying \eqref{eq:certificate-domain}, one has
	\[
		\Theta_0\geq0,\qquad \Theta_1>0.
	\]
	Moreover, \(\Theta_0=0\) if and only if \(t=\frac{1}{2}\) and \(f_3=0\).
\end{lemma}

\begin{proof}
	For fixed \(n\) and \(\frac{f_3^2}{S^3}\), we separate the value of \(\frac{\Theta_0}{S^3}\) at \(t=\frac{1}{2}\) from the remaining terms.
	Substituting \eqref{eq:integral-f-psi} into \eqref{eq:Theta-definition} and factoring \(\frac{1}{2}-t\) from the difference yields
	\begin{equation}\label{eq:Theta-constant}
		\begin{aligned}
			\frac{\Theta_0}{S^3}
			=-\frac{\frac{f_3^2}{S^3}P_e\left(n,\frac{f_3^2}{S^3}\right)}{128n(n+2)\left(4+\frac{nf_3^2}{S^3}\right)^2}
			+\frac{\left(\frac{1}{2}-t\right)P_c\left(t,n,\frac{f_3^2}{S^3}\right)}{192n^2(3-4t)^2\left(4+\frac{nf_3^2}{S^3}\right)^2}.
		\end{aligned}
	\end{equation}
	Here \(P_e\) and \(P_c\) are given by
	\begin{equation}\label{eq:Pe}
		\begin{aligned}
			-P_e\left(n,\frac{f_3^2}{S^3}\right)
			={}&64(n-2)(n+3)
			+\frac{16nf_3^2}{S^3}(3n^2-3n-26)\\
			&+\frac{4f_3^4}{S^6}(2n^4-7n^3-26n^2)
			-\frac{5n^3(n+2)f_3^6}{S^9},
		\end{aligned}
	\end{equation}
	and
	\begin{equation}\label{eq:Pc}
		\begin{aligned}
			P_c\left(t,n,\frac{f_3^2}{S^3}\right)
			={}&A_0+48\left(\frac{1}{2}-t\right)A_1
			-768\left(\frac{1}{2}-t\right)^2
			\left(\frac{f_3^2}{S^3}+\frac{2}{n}\right)A_2\\
			&-4096\left(\frac{1}{2}-t\right)^3
			\left(\frac{nf_3^2}{S^3}+2\right)^2(n+1)(n+3),
		\end{aligned}
	\end{equation}
	where
	\begin{equation}\label{eq:algebraic-A0}
		\begin{aligned}
			A_0={}&\left(\frac{nf_3^2}{S^3}+2\right)^2
			\left[128n^2+84n\left(\frac{nf_3^2}{S^3}+2\right)
			-9\left(\frac{nf_3^2}{S^3}+2\right)^2\right]\\
			&+4\left(\frac{nf_3^2}{S^3}+2\right)^2
			\left[206n+63\left(\frac{nf_3^2}{S^3}+2\right)+336\right]\\
			&+144\left[3\left(\frac{nf_3^2}{S^3}+2\right)(n+3)+2n+17\right],
		\end{aligned}
	\end{equation}
	\begin{equation}\label{eq:algebraic-A1}
		\begin{aligned}
			A_1={}&\left(\frac{nf_3^2}{S^3}+2\right)^2
			\left[16n^2-n\left(\frac{nf_3^2}{S^3}+2\right)
			-3\left(\frac{nf_3^2}{S^3}+2\right)^2\right]\\
			&+\left(\frac{nf_3^2}{S^3}+2\right)^2
			\left[74n-9\left(\frac{nf_3^2}{S^3}+2\right)\right]\\
			&+58\left(\frac{nf_3^2}{S^3}+2\right)^2
			+12\left(\frac{nf_3^2}{S^3}+2\right)(3n-1)+24(n+7),
		\end{aligned}
	\end{equation}
	and
	\begin{equation}\label{eq:algebraic-A2}
		A_2=\left(\frac{nf_3^2}{S^3}+2\right)^2(2n^2+3n)
		+4n\left(\frac{nf_3^2}{S^3}+2\right)(n+1)+12n.
	\end{equation}
	Here \(A_0,A_1,A_2\) are polynomial coefficients.

	By \eqref{eq:Pe}, the polynomial \(s\mapsto -P_e(n,s)\) has positive constant and linear coefficients and negative leading coefficient for \(n\geq5\), since
	\[
		64(n-2)(n+3)>0,\qquad
		16n(3n^2-3n-26)>0,\qquad
		-5n^3(n+2)<0.
	\]
	Lemma~\ref{lem:cubic-endpoint}, with \(L=\frac{n-2}{n}\), therefore yields
	\[
		\min_{0\leq s\leq\frac{n-2}{n}}
		\bigl\{-P_e(n,s)\bigr\}
		=
		\min\left\{
		-P_e(n,0),
		-P_e\left(n,\frac{n-2}{n}\right)
		\right\}.
	\]
	Evaluating \eqref{eq:Pe} at the two endpoints gives \(-P_e(n,0)=64(n-2)(n+3)>0\) and \(-P_e\left(n,\frac{n-2}{n}\right)=(n-2)^2(n+2)(3n+14)>0\).
	Consequently,
	\begin{equation}\label{eq:Pe-sign}
		-P_e\left(n,\frac{f_3^2}{S^3}\right)>0
	\end{equation}
	for the parameter range in \eqref{eq:certificate-domain}.

	The bounds in \eqref{eq:certificate-domain} also yield
	\[
		2\leq\frac{nf_3^2}{S^3}+2\leq n,
		\qquad
		n\leq\frac{n}{2}\left(\frac{nf_3^2}{S^3}+2\right).
	\]
	These bounds imply \(84n\left(\frac{nf_3^2}{S^3}+2\right)-9\left(\frac{nf_3^2}{S^3}+2\right)^2\geq0\), so the first bracket in \eqref{eq:algebraic-A0} is at least \(128n^2\).
	The remaining terms are nonnegative, and hence
	\begin{equation}\label{eq:A0-bound}
		A_0\geq128n^2\left(\frac{nf_3^2}{S^3}+2\right)^2.
	\end{equation}
	The two brackets in \eqref{eq:algebraic-A1} satisfy
	\[
		\begin{aligned}
			16n^2-n\left(\frac{nf_3^2}{S^3}+2\right)
			-3\left(\frac{nf_3^2}{S^3}+2\right)^2&\geq12n^2,\\
			74n-9\left(\frac{nf_3^2}{S^3}+2\right)&\geq65n.
		\end{aligned}
	\]
	The remaining terms in \eqref{eq:algebraic-A1} are positive, so \(A_1>0\).
	Using \(\frac{1}{n}\leq\frac{1}{5}\) and \(\frac{nf_3^2}{S^3}+2\geq2\) in \eqref{eq:algebraic-A2} gives
	\begin{equation}\label{eq:A2-bound}
		\begin{aligned}
			A_2
			\leq\left(\frac{13}{5}+\frac{12}{5}+\frac{3}{5}\right)
			n^2\left(\frac{nf_3^2}{S^3}+2\right)^2
			=\frac{28}{5}n^2\left(\frac{nf_3^2}{S^3}+2\right)^2.
		\end{aligned}
	\end{equation}
	The second and third estimates use \(4n(n+1)\left(\frac{nf_3^2}{S^3}+2\right)\leq2n(n+1)\left(\frac{nf_3^2}{S^3}+2\right)^2\) and \(12n\leq3n\left(\frac{nf_3^2}{S^3}+2\right)^2\), respectively.

	For fixed \(n\) and \(\frac{f_3^2}{S^3}\), the expression for \(P_c\) in \eqref{eq:Pc} is a cubic polynomial in \(\frac{1}{2}-t\), with positive constant and linear coefficients and a negative leading coefficient.
	By Lemma~\ref{lem:cubic-endpoint}, its minimum on \(0\leq\frac{1}{2}-t\leq\frac{1}{10}\) is attained at \(t=\frac{1}{2}\) or \(t=\frac{2}{5}\).
	The value at \(t=\frac{1}{2}\) is \(A_0>0\).
	At \(t=\frac{2}{5}\), the bounds \eqref{eq:A0-bound} and \eqref{eq:A2-bound}, together with \(A_1>0\), yield
	\[
		\begin{aligned}
			P_c\left(\frac{2}{5},n,\frac{f_3^2}{S^3}\right)
			={}&A_0+\frac{24}{5}A_1
			-\frac{192}{25}\left(\frac{f_3^2}{S^3}+\frac{2}{n}\right)A_2
			-\frac{512}{125}\left(\frac{nf_3^2}{S^3}+2\right)^2(n+1)(n+3)\\
			>{}&\left(128-\frac{192}{25}\cdot\frac{28}{5}
			-\frac{512}{125}\cdot\frac{48}{25}\right)
			n^2\left(\frac{nf_3^2}{S^3}+2\right)^2\\
			={}&\frac{241024}{3125}n^2\left(\frac{nf_3^2}{S^3}+2\right)^2>0.
		\end{aligned}
	\]
	Here we use \(\frac{f_3^2}{S^3}+\frac{2}{n}\leq1\), \((n+1)(n+3)\leq\frac{48}{25}n^2\) and \(A_1>0\).
	Thus
	\begin{equation}\label{eq:Pc-sign}
		P_c\left(t,n,\frac{f_3^2}{S^3}\right)>0
	\end{equation}
	throughout \eqref{eq:certificate-domain}.
	By \eqref{eq:Pe-sign} and \eqref{eq:Pc-sign}, both terms in \eqref{eq:Theta-constant} are nonnegative.
	They vanish precisely when \(f_3=0\) and \(t=\frac{1}{2}\), respectively.
	Thus \(\Theta_0\geq0\), with equality if and only if both conditions hold.

	Substituting \eqref{eq:integral-f-psi} into \eqref{eq:Theta-definition} and collecting the coefficient of \(\int_M\psi\) yields
	\begin{equation}\label{eq:Theta-slope}
		\begin{aligned}
			\Theta_1={}&\frac{t}{3-4t}\left(\frac{S}{n}+\frac{f_3^2}{4S^2}\right)
			+\frac{(n+4)S}{4n}-\frac{5f_3^2}{16S^2}
			-\frac{3}{2}\left(\frac{3}{5}+\frac{t^2}{(3-4t)(1-t)}\right)
			\left(\frac{S}{n}+\frac{f_3^2}{4S^2}\right)\\
			&+\frac{(n+4)S(nf_3^2+2S^3)^2}{n(nf_3^2+4S^3)^2}
			-\frac{S(nf_3^2+2S^3)\bigl((n+4)(3-4t)+20t\bigr)}{5n(nf_3^2+4S^3)}.
		\end{aligned}
	\end{equation}
	Dividing \eqref{eq:Theta-slope} by \(S\) and clearing denominators, we find that
	\begin{equation}\label{eq:slope-polynomial}
		80n(1-t)(3-4t)\left(\frac{nf_3^2}{S^3}+4\right)^2\frac{\Theta_1}{S}
	\end{equation}
	is a cubic polynomial in \(\frac{f_3^2}{S^3}\).
	Its constant coefficient is
	\begin{equation}\label{eq:slope-constant-coefficient}
		32\left[
		16n\left(\frac{1}{2}+t\right)(1-t)(3-4t)
		-\left(64t^3+4t^2-20t+12\right)
		\right],
	\end{equation}
	and its linear coefficient is
	\begin{equation}\label{eq:slope-linear-coefficient}
		16\left[
		24n^2\left(\frac{1}{2}+t\right)(1-t)(3-4t)
		-n\left(96t^3+106t^2-205t+93\right)
		\right].
	\end{equation}
	Its leading coefficient is \(-3n^3\left(74t^2-107t+43\right)<0\).
	On the interval \(\frac{2}{5}\leq t\leq\frac{1}{2}\),
	\[
		192t^2+8t-20\geq\frac{348}{25}>0,
		\qquad
		288t^2+212t-205\leq-27<0.
	\]
	Thus \(64t^3+4t^2-20t+12\) is increasing, whereas \(96t^3+106t^2-205t+93\) is decreasing.
	Hence
	\[
		64t^3+4t^2-20t+12\leq11,
		\qquad
		96t^3+106t^2-205t+93\leq\frac{4263}{125}<35.
	\]
	Moreover,
	\[
		2\left(\frac{1}{2}+t\right)(1-t)(3-4t)
		=1+(2t-1)(4t^2-3t-2)\geq1.
	\]
	These estimates bound the brackets in \eqref{eq:slope-constant-coefficient} and \eqref{eq:slope-linear-coefficient} below by \(8n-11>0\) and \(12n^2-35n>0\), respectively.
	Lemma~\ref{lem:cubic-endpoint} therefore reduces the positivity of \eqref{eq:slope-polynomial} to its values at \(\frac{f_3^2}{S^3}=0\) and \(\frac{f_3^2}{S^3}=\frac{n-2}{n}\).
	The value at the first endpoint is positive by \eqref{eq:slope-constant-coefficient}.

	At the second endpoint, substitution into \eqref{eq:Theta-slope} gives
	\begin{equation}\label{eq:slope-endpoint-bound}
		\begin{aligned}
			&\left.
			80n(1-t)(3-4t)\left(\frac{nf_3^2}{S^3}+4\right)^2\frac{\Theta_1}{S}
			\right|_{\frac{f_3^2}{S^3}=\frac{n-2}{n}}\\
			\geq&15n^3-148n
			+\frac{1-2t}{2}\bigl[103n^3+2n(119n-142)\bigr]\\
			&-\frac{(1-2t)^2}{4}\bigl[102n^3+4n(93n+290)\bigr]
			-32(1-2t)^3n^2(n+1).
		\end{aligned}
	\end{equation}
	Indeed, subtracting the right-hand side of \eqref{eq:slope-endpoint-bound} from the left-hand side yields
	\[
		\begin{aligned}
			&3n^3+18n^2+36n+168\\
			&\quad +(1-2t)(9n^3+54n^2+108n+684)\\
			&\quad +(1-2t)^2(6n^3+36n^2+72n+276)
			+64(1-2t)^3n,
		\end{aligned}
	\]
	which is nonnegative.

	For \(n\geq5\), we have
	\[
		\begin{aligned}
			15n^3-148n&>9n^3,\\
			103n^3+2n(119n-142)&\geq103n^3,\\
			102n^3+4n(93n+290)&<223n^3,\\
			32n^2(n+1)&<39n^3.
		\end{aligned}
	\]
	Using these bounds and \(0\leq1-2t\leq\frac{1}{5}\) in \eqref{eq:slope-endpoint-bound} gives
	\[
		\begin{aligned}
			&\left.
			80n(1-t)(3-4t)\left(\frac{nf_3^2}{S^3}+4\right)^2\frac{\Theta_1}{S}
			\right|_{\frac{f_3^2}{S^3}=\frac{n-2}{n}}\\
			>&9n^3
			+\frac{1-2t}{4}n^3
			\left[206-223(1-2t)-156(1-2t)^2\right]\\
			\geq&\left[9+\frac{155}{4}(1-2t)\right]n^3>0.
		\end{aligned}
	\]
	Both endpoint values are positive, so Lemma~\ref{lem:cubic-endpoint} implies that \eqref{eq:slope-polynomial} is positive throughout the parameter range.
	Since \(S>0\) and the remaining factors in \eqref{eq:slope-polynomial} are positive, \(\Theta_1>0\).
\end{proof}

Lemmas~\ref{lem:integral-comparison} and~\ref{lem:algebraic-positivity}, together with Proposition~\ref{prop:hessian-cancellation}, yield the following rigidity statement.

\begin{proposition}\label{prop:closed-interval-rigidity}
	Under the hypotheses of Theorem~\ref{thm:main}, suppose that \(\frac{5}{3}n\leq S\leq2n\).
	Then \(S=2n\), \(f_3=0\), and \(h^3=3h\).
\end{proposition}

\begin{proof}
	The hypotheses imply the parameter bounds in \eqref{eq:certificate-domain}.
	By \eqref{eq:Theta-sign}, \(\Theta\leq0\), whereas Lemma~\ref{lem:algebraic-positivity}, \eqref{eq:psi-range-integral}, and \eqref{eq:Theta-affine} give
	\[
		\Theta=\Theta_0\Vol(M)+\Theta_1\int_M\psi\geq0.
	\]
	Thus equality holds, and \(\Vol(M)>0\) and \(\Theta_1>0\) force \(\Theta_0=0\) and \(\int_M\psi=0\).
	Lemma~\ref{lem:algebraic-positivity} then gives \(t=\frac{1}{2}\) and \(f_3=0\), so \(S=2n\).
	The vanishing of \(\int_M\psi\), together with \(\psi\geq0\) in \eqref{eq:psi-range-integral}, implies \(\psi\equiv0\).
	With \(\psi\equiv0\), identity~\eqref{eq:r-orthogonality} shows that the integrand in \eqref{eq:sharp-residual-bound} is \(\sum\limits_{i,j}r_{ij}^2\).
	Hence \eqref{eq:sharp-residual-bound} yields
	\[
		\int_M\sum_{i,j}r_{ij}^2=0,
	\]
	and therefore \(r_{ij}\equiv0\).
	At \(t=\frac{1}{2}\) and \(f_3=0\), \eqref{eq:k-polynomial} reduces to
	\[
		0=r_{ij}=\sum_{p,q}h_{ip}h_{pq}h_{qj}-3h_{ij}.
	\]
	Thus \(h^3=3h\).
\end{proof}

\section{Proof of Theorem~\ref{thm:main}}
\label{sec:proof-main}

Yang--Cheng's estimate and Proposition~\ref{prop:closed-interval-rigidity} exclude the interval \(n<S<2n\).
At \(S=2n\), the identity \(h^3=3h\) and minimality determine the principal curvatures and their multiplicities, so Cartan's classification identifies the image.
The simple connectivity of the resulting Cartan hypersurface then allows us to conclude that the immersion is an embedding.

\begin{proof}
	Assume that \(S>n\).
	If \(S<2n\), then \cite[Theorem~2]{yang_cheng_1998} gives \(S\geq\frac{5}{3}n\).
	Proposition~\ref{prop:closed-interval-rigidity} then implies \(S=2n\), a contradiction.
	Therefore \(S\geq2n\).

	Suppose now that \(S=2n\).
	For the curvature calculations, we pass to a connected component \(\widehat M\) of the two-sided cover of \(M\) and fix a unit normal.
	Proposition~\ref{prop:closed-interval-rigidity} gives \(h(h^2-3I)=0\), so every principal curvature belongs to \(\{-\sqrt3,0,\sqrt3\}\).
	Let \(m_-\), \(m_0\), and \(m_+\) denote the corresponding multiplicities at a point.
	Minimality, the identity \(S=2n\), and the dimension relation yield
	\[
		\sqrt3(m_+-m_-)=0,\qquad
		3(m_++m_-)=2n,\qquad
		m_-+m_0+m_+=n.
	\]
	These equations give \(m_-=m_0=m_+=\frac{n}{3}\).
	The three principal curvatures therefore occur at every point with constant multiplicities, so the lifted immersion of \(\widehat M\) is isoparametric with \(g=3\).

	Cartan's classification gives \(m\in\{1,2,4,8\}\) for the common multiplicity.
	Up to a fixed ambient isometry, the image of the connected lifted immersion is an open subset of a tube of constant radius about the standard embedding of \(\mathbb F P^2\), where \(\mathbb F=\mathbb R,\mathbb C,\mathbb H,\mathbb O\) corresponds to \(m=1,2,4,8\), respectively \cite[pp.~359--367]{cartan_isoparametric_1939}; see also \cite{cecil_ryan_2015}.
	Since \(m=\frac{n}{3}\) and \(n\geq5\), only \(m=2,4,8\) are possible, giving \(n=6,12,24\), respectively.
	In each family, the principal curvatures \(\sqrt3,0,-\sqrt3\) determine the minimal member.

	Let \(C_m\) denote this connected minimal Cartan hypersurface.
	The image of \(\widehat M\) is open in \(C_m\) and is closed by compactness.
	Since \(C_m\) is connected, this image is all of \(C_m\).
	The projection \(\widehat M\to M\) is surjective, so the original immersion also has image \(C_m\).
	As a map \(M\to C_m\), this immersion is a proper local diffeomorphism and hence a finite covering.
	For \(m=2,4,8\), \(C_m\) is diffeomorphic to the unit normal sphere bundle of \(\mathbb F P^2\).
	Under this identification, the focal map is the projection in the sphere bundle
	\[
		\mathbb S^m(1)\longrightarrow C_m\longrightarrow\mathbb F P^2,
		\qquad \mathbb F=\mathbb C,\mathbb H,\mathbb O,
	\]
	respectively \cite{cecil_ryan_2015}.
	Both the fiber and the base are simply connected, so the homotopy exact sequence of this bundle yields \(\pi_1(C_m)=0\).
	The connected covering \(M\to C_m\) therefore has one sheet and is a diffeomorphism, which proves that the original immersion is an embedding.
	Normal reversal interchanges \(\sqrt3\) and \(-\sqrt3\) and fixes \(0\), so the unordered spectrum is independent of the chosen normal.

	Conversely, if the image is congruent to a minimal Cartan hypersurface with common multiplicity \(m\in\{2,4,8\}\), then \(n=3m\) and the principal curvatures are \(\sqrt3,0,-\sqrt3\), each of multiplicity \(m\).
	Hence
	\[
		S=m(3+0+3)=6m=2n,\qquad
		f_3=m(3\sqrt3+0-3\sqrt3)=0.
	\]
	This proves the converse and completes the proof.
\end{proof}

\paragraph{AI Disclosure.}
ChatGPT was used only to assist in verifying some algebraic calculations and was consulted on the choice of certain coefficients.
The authors independently developed and verified all mathematical arguments and the resulting proofs, prepared the final manuscript, and take full responsibility for its content.

\end{document}